\documentclass[12pt]{amsart}
\numberwithin{equation}{section}

\usepackage{amssymb}
\def\cb{{\mathcal B}}

\def\ce{{\mathcal E}}

\def\ch{{\mathcal H}}

\def\ck{{\mathcal K}}

\def\cs{{\mathcal S}}

\def\ga{{\mathfrak A}} \def\gpa{{\mathfrak a}}
 \def\gpb{{\mathfrak b}}

\def\gpm{{\mathfrak m}}

 \def\gpo{{\mathfrak o}}

\def\bc{{\mathbb C}}
\def\bd{{\mathbb D}}

\def\bn{{\mathbb N}}

\def\bt{{\mathbb T}}

\def\bz{{\mathbb Z}}

\def\a{\alpha}

 \def\G{\Gamma}
\def\d{\delta}

\def\l{\lambda} 

\def\m{\mu}

\def\s{\sigma} 

\def\f{\varphi} \def\F{\Phi}
\def\th{\theta}  
\def\om{\omega}

\def\id{\hbox{id}}

\newtheorem{thm}{Theorem}[section]
\newtheorem{lem}[thm]{Lemma}

\newtheorem{prop}[thm]{Proposition}

\newtheorem{rem}[thm]{Remark}

\def\ad{\mathop{\rm Ad}}

\def\di{\mathop{\rm d}\!}

\def\idd{{1}\!\!{\rm I}}

\begin{document}

\title[uniform convergence]
{on the uniform convergence of ergodic averages for $C^*$-dynamical systems}
\author{Francesco Fidaleo}
\address{Dipartimento di Matematica,
II Universit\`{a} di Roma ``Tor Vergata'',
Via della Ricerca Scientifica, 00133 Roma, Italy.
e--mail: \tt fidaleo@mat.uniroma2.it}

\date{\today}

\begin{abstract}
We investigate some ergodic and spectral properties of general (discrete) $C^*$-dynamical systems $(\ga,\F)$ made of a unital $C^*$-algebra and a multiplicative, identity-preserving $*$-map $\F:\ga\to\ga$, particularising the situation when $(\ga,\F)$ enjoys the property of unique ergodicity with respect to the fixed-point subalgebra. 

For $C^*$-dynamical systems enjoying or not the strong ergodic property mentioned above, we provide conditions on $\l$ in the unit circle $\{z\in\bc\mid |z|=1\}$ and the corresponding eigenspace $\ga_\l\subset\ga$ for which the sequence of Cesaro averages $\left(\frac1{n}\sum_{k=0}^{n-1}\l^{-k}\F^k\right)_{n>0}$, converges point-wise in norm.

We also describe some pivotal examples coming from quantum probability, to which the obtained results can be applied.

\end{abstract}

\maketitle

\section{introduction}

The present paper is devoted to the investigation of ergodic properties of noncommutative ({\it i.e.} quantum) $C^*$-dynamical systems.

The investigation of ergodic properties of classical dynamical systems was at first motivated by the problem of justifying the thermodynamical laws from the microscopic principles of statistical mechanics, {\it i.e.} the so-called ergodic hypothesis. However, after discovering the quantum behaviour of the matter at the microscopic level ({\it i.e.} the quantum mechanics), it was natural to address the systematic study of noncommutative aspects of many branches of mathematics, which has seen an impetuous growth in the last decades, not least because of their applications to quantum physics.

Concerning classical dynamical systems $(X,T,\m)$ made of a compact Hausdorff space $X$, a continuous map $T:X\rightarrow X$, and a probability Radon measure  $\m$  which is invariant under the action of $T$ ({\it i.e.} $\m\circ T^{-1}=\m$), classical ergodic theory primarily deals with the study of the long-time behaviour of the Cesaro means ({\it i.e.} ergodic~averages)
\begin{equation}
\label{anl}
M_{f,\l}(n):=\frac1{n}\sum_{k=0}^{n-1}\l^{-k}f\circ T^{k}\,,\quad n=0,1\dots\,,
\end{equation}
$|\l|=1$, of continuous functions or more generally of any measurable function $f$.

Among such results which constitute the milestones of ergodic theory, we mention the Birkhoff individual ergodic theorem, which concerns the study of the point-wise limit
$\lim_{n\to+\infty}M_{f,1}(n)(x)$, $x\in X$, when $f$ is summable, and the von Neumann mean 
ergodic theorem, dealing with the limit 
$L^2$-$\lim_{n\to+\infty}M_{f,1}(n)$ when $f$ is square-summable. 

At the same way, also the investigation of uniform convergence of ergodic averages ({\it i.e.} involving directly continuous functions and the convergence in the norm topology) in \eqref{anl} is of great interest. We mention the following situation relative to dynamical systems enjoying a very strong ergodic property. Indeed, the classical dynamical system $(X,T)$ is said to be uniquely ergodic
if there exists a unique probability Radon measure $\m$ which is invariant under the action of the transformation $T$. It is a well-known fact that $(X,T)$ is uniquely ergodic if and only if, for the Cesaro average in \eqref{anl},
$$
\lim_{n\to+\infty}M_{f,1}(n)=\int_X f\di\m\,,\quad f\in C(X)\,,
$$
uniformly. In \cite{R}, the limit of the Cesaro averages $M_{f,\l}$ was investigated for more general $\l$ in the unit circle.

In view of natural applications to quantum physics, it is then natural to address the systematic study of the ergodic properties of quantum dynamical systems. As a matter of fact, the situation in the quantum setting is rather more involved than the classical situation in several respects. For instance, all statements must be provided in terms of the dual concept of ``functions'' instead of ``points'', since the latter are meaningless in the quantum cases.

As for the literature on noncommutative ergodic theory, 
the reader is firstly referred to the seminal paper \cite{NSZ} for a thorough study of the multiple correlations and quantum 
(weak) mixing associated with invariant states with central support in the bidual. Some natural generalisations of quantum ergodic theory are investigated in a series of  papers \cite{F, F1, F2, F3, F4} without assuming in general the centrality of the support of the involved states, whereas the reader is referred to \cite{BF2, CrF} for some direct applications to physics and quantum probability.

In the paper \cite{F4}, the generalisation of the result in \cite{R} was extended to the noncommutative setting. More precisely, let $(\ga, \F,\f)$ be a uniquely ergodic $C^*$-dynamical system based on a unital $C^*$-algebra, a unital $*$-homomorphism $\F:\ga\to\ga$ with $\f\in\cs(\ga)$ as the unique 
invariant state. Consider the covariant Gelfand-Naimark-Segal representation $\big(\ch_\f,\pi_\f, V_{\f,\F},\xi_\f\big)$ associated to the state $\f$, together with the peripheral pure-point spectra (see below for the definition) $\s^{\mathop{\rm ph}}_{\mathop{\rm pp}}(\F)$ and $\s^{\mathop{\rm ph}}_{\mathop{\rm pp}}(V_{\f,\F})$ of $\F$ and the isometry $V_{\f,\F}\in\cb(\ch_\f)$, respectively. We have 
$\s^{\mathop{\rm ph}}_{\mathop{\rm pp}}(\F)\subset\s^{\mathop{\rm ph}}_{\mathop{\rm pp}}(V_{\f,\F})$, but in general they are different. As for the classical case, we have shown that the analogous 
\begin{equation}
\label{malam}
M_{a,\l}(n):=\frac1{n}\sum_{k=0}^{n-1}\l^{-k}\F^{k}(a)\,,\quad n\in\bn\smallsetminus\{0\}
\end{equation}
of the Cesaro averages in \eqref{anl} converge in norm for each fixed $a\in\ga$ and 
$\l\in\s^{\mathop{\rm ph}}_{\mathop{\rm pp}}(\F)\bigcup\s^{\mathop{\rm ph}}_{\mathop{\rm pp}}(V_{\f,\F})^{\rm c}$, where the complement is taken in the unit circle $\bt:=\{\l\in\bc\mid|\l|=1\}$. 

Some examples based on noncommutative 2-torus relative to the non convergence of $M_{a,\l}(n)$ for $\l\in\s^{\mathop{\rm ph}}_{\mathop{\rm pp}}(V_{\f,\F})\smallsetminus\s^{\mathop{\rm ph}}_{\mathop{\rm pp}}(\F)$ are also exhibited
in \cite{DFGR}.

In the present paper, we generalise some results obtained in \cite{F4} for $C^*$-dynamical systems $(\ga, \F)$ as above. Indeed, we first consider the set $\cs(\ga)^\F$ made of all invariant states under the action of the $*$-endomorphism $\F$, and
define the {\it full peripheral pure-point spectrum} as
$$
\s^{(\mathop{\rm ph, f})}_{\mathop{\rm pp}}(\F):=\bigcup\big\{\s^{\mathop{\rm ph}}_{\mathop{\rm pp}}(V_{\f,\F})\mid \f\in\cs(\ga)^\F\big\}\,.
$$ 
Notice that, it is a spectral set canonically associated to the $C^*$-dynamical system $(\ga, \F)$.
 
The first noteworthy result we can prove is that, if $\l\in\bt\smallsetminus\s^{(\mathop{\rm ph, f})}_{\mathop{\rm pp}}(\F)$, then for each $a\in\ga$,
$$
\lim_{n\to+\infty}\frac1{n}\sum_{k=0}^{n-1}\l^{-k}\F^k(a)=0\,,
$$
in the norm topology, or "uniformly" with an abuse of terminology, but in analogy with the classical case. 

In general, we have no natural condition which assures the convergence of the averages $M_{a,\l}(n)$ for $\l\in\s^{\mathop{\rm ph}}_{\mathop{\rm pp}}(\F)$ as it happens for uniquely ergodic $C^*$-dynamical systems. However, some partial results can be also obtained under the assumption of a weaker property of unique ergodicity. 

Namely, we consider the $C^*$-dynamical systems $(\ga, \F)$ which are {\it uniquely ergodic with respect to the fixed-point subalgebra}. The last condition of ergodicity is weaker than unique ergodicity, provided that for the fixed-point $*$-subalgebra, $\ga^\F\supsetneq\bc\idd_{\ga}$. 

For such systems, and for the eigenvalues $\l\in\s^{\mathop{\rm ph}}_{\mathop{\rm pp}}(\F)$ for which the associated eigenspaces $\ga_\l$ contain an isometry or a co-isometry, we can show that there exists a norm one projection $E_\l:\ga\rightarrow\ga_{\l}$ such that
$$
\lim_{n\to+\infty}\frac1{n}\sum_{k=0}^{n-1}\l^{-k}\F^k=E_\l\,,
$$
point-wise in norm. Such a result includes that corresponding to uniquely ergodic dynamical systems $(\ga,\F.\f)$ for which, when $\l\in\s^{\mathop{\rm ph}}_{\mathop{\rm pp}}(\F)$, the eigenspaces $\ga_\l$ are automatically generated by a single unitary and 
$\f(xu^*)u=E_\l(x)=\f(u^*x)u$.

The paper ends with some examples arising from quantum probability, enjoying or not unique ergodicity with respect to the fixed-point subalgebra, for which the results obtained in the present paper can be applied. More complicated examples arising from the noncommutative geometry ({\it i.e.} the noncommutative 2-torus, see \cite{DFGR}) will be presented elsewhere.

\section{preliminaries}

In the present paper, we deal without further mention with unital $C^*$-algebras $\ga$ with unity $0\neq\idd:=\idd_\ga$. We also recall that $u\in\ga$ is an isometry (co-isometry) if $u^*u=\idd$ ($uu^*=\idd$). A unitary operator is an element $u\in\ga$ which is an isometry and a co-isometry: $u^*u=\idd=uu^*$.

With $\bd:=\{\l\in\bc\mid|\l|\leq1\}$ and $\bt:=\{\l\in\bc\mid|\l|=1\}$ we denote the unit disc and the unit circle of the complex plane, respectively. Of course, $\bt=\partial\bd$. The unit circle $\bt$ is homeomorphic to the interval $[0,2\pi)$ by
$\th\in[0,2\pi)\mapsto e^{-\imath\th}$, after identifying the endpoints $0$ and $2\pi$.

A (discrete) $C^*$-dynamical system is a pair $(\ga,\F)$ consisting of a $C^*$-algebra and a positive map $\F:\ga\to\ga$. Notice that, if $\|\F\|=1$, which happens if $\F$ is completely positive and $\|\F(\idd)\|=1$, then $\s(\F)\subset\bd$. The part of the spectrum $\s(\F)\bigcap\bt$ living on the unit circle is called {\it peripheral}.

We denote by $\cs(\ga)^\F$ and $\partial\left(\cs(\ga)^\F\right)$ the convex $*$-weakly compact set of the invariant states of $\ga$ under the action of $\F$, and its convex boundary made of extreme invariant states, respectively. The elements of $\partial\left(\cs(\ga)^\F\right)$ are called the {\it ergodic states}, see {\it e.g.} \cite{Sa}.

With a slight abuse of notation, we also denote the triplet $(\ga,\F,\f)$ as a $C^*$-dynamical system when we want to point out the state $\f$ on 
$\ga$, which is invariant under the dynamics generated by $\F$. 

For the $C^*$-dynamical system $(\ga,\F,\f)$ as above, consider the Gelfand-Naimark-Segal (GNS for short)  representation $\big(\ch_\f,\pi_\f,\xi_\f\big)$, see {\it e.g.} \cite{Sa}. If in addition
$$
\f\big(\F(a)^*\F(a)\big)\leq \f(a^*a)\,,\quad a\in\ga\,,
$$
then there exists a unique contraction $V_{\f,\F}\in\cb(\ch_\f)$ such that $V_{\f,\F}\xi_\f=\xi_\f$ and
$$
V_{\f,\F}\pi_\f(a)\xi_\f=\pi_\f(\F(a))\xi_\f\,,\quad a\in\ga\,.
$$
The quadruple $\big(\ch_\f,\pi_\f, V_{\f,\F},\xi_\f\big)$ is called the {\it covariant GNS representation} associated to $(\ga,\F,\f)$.

If $\F$ is multiplicative, hence a $*$-homomorphism, then $V_{\f,\F}$ is an isometry with range-projection $V_{\f,\F}V_{\f,\F}^*$, the orthogonal projection onto the subspace $\overline{\pi_\f\big(\F(\ga)\big)\xi_\f}$, see \cite{NSZ}, Lemma 2.1.

Concerning the $C^*$-dynamical system $(\ga,\F)$ made of a unital $C^*$-algebra $\ga$ and a unital completely positive linear map $\F:\ga\to\ga$, we can consider the sequence \eqref{malam} of the Cesaro averages
$$
M_{a,\l}(n)=\frac1{n}\sum_{k=0}^{n-1}\l^{-k}\F^{k}(a)\,,\quad \l\in\bc\smallsetminus\{0\}\,,\,\,\,a\in\ga\,.
$$
\begin{rem}
Notice that:
\begin{itemize}
\item[(i)] if $|\l|>1$, then $\lim_{n\to+\infty}M_{a,\l}(n)=0$ in the norm topology;
\item[(ii)] if $M_{a,\l}(n)$ converges in the norm topology (resp. in the $\s(\ga,\ga^*)$-topology), then
$$
\lim_{n\to+\infty}\frac{\l^{-n}\F^{n}(a)}{n}=0\,,
$$
in the norm topology (resp. in the $\s(\ga,\ga^*)$-topology).
\end{itemize}
\end{rem}
Indeed, we have
$$
\big\|M_{a,\l}(n)\big\|\leq\Big(\frac{\|a\|}{1-|\l|^{-1}}\Big)\frac{1-|\l|^{-n}}{n}\rightarrow0\,,
$$
if $|\l|>1$ and $n\to+\infty$.

On the other hand,
$$
M_{a,\l}(n+1)-M_{a,\l}(n)=\frac{\l^{-(n+1)}}{n+1}\F^{(n+1)}(a)-\frac{M_{a,\l}(n)}{n+1}\,,
$$
and therefore
$$
M_{a,\l}(n+1)-\Big(1-\frac1{n+1}\Big)M_{a,\l}(n)=\frac{\l^{-(n+1)}}{n+1}\F^{(n+1)}(a)\,.
$$
We conclude that, if $M_{a,\l}(n)$ converges in some topology of $\ga$ then, necessarily, $\frac{\l^{-n}\F^{n}(a)}{n}\to0$ in the corresponding topology.

\noindent\hspace{12.33 cm}$\square$

For $0<|\l|<1$, we then argue that, if $\limsup_n\big(|\l|^{-n}\|\F^{n}(a)\|/n\big)>0$ (which happens if $\F$ is isometric and $a\neq0$) or there exists $f\in\ga^*$ such that $\limsup_n\big(|\l|^{-n}|f(\F^{n}(a))|/n\big)>0$, then $M_{a,\l}(n)$ cannot converge in the norm topology or in the weak topology, respectively. 

Therefore, the most complicated task concerning the convergence of the sequence $\left(\frac1{n}\sum_{k=0}^{n-1}\l^{-k}\F^k(a)\right)_{n>0}$ for fixed $a\in\ga$ and $\l\in\bc\smallsetminus\{0\}$, corresponds to the situation $|\l|=1$, which is the topic of the present paper.

Among all dynamical systems, we can consider those enjoying some strong ergodic properties.  Indeed, a $C^*$-dynamical system $(\ga,\F)$ made of a unital $C^*$-algebra $\ga$ and a unital completely positive linear map $\F:\ga\to\ga$ is said to be {\it uniquely ergodic} if there exists only one invariant state $\f$ for the dynamics induced by $\F$. For a uniquely ergodic $C^*$-dynamical system, we simply write $(\ga,\F,\f)$ by pointing out that $\f\in\cs(\ga)$ is the unique invariant state. In this case,
$$
\cs(\ga)^\F=\partial\left(\cs(\ga)^\F\right)=\{\f\}\,.
$$

Denote by $\ga^\F:=\big\{a\in\ga\mid\F(a)=a\big\}$ the fixed-point subspace. If $\F$ is multiplicative, $\ga^\F$ is a $*$-subalgebra. If $\F$ is merely completely positive, then $\ga^\F$ is an {\it operator system} (see {\it e.g.} \cite{P} for the definition of operator system).

The $C^*$-dynamical system $(\ga,\F)$, with $\F$ multiplicative, is said to be {\it uniquely ergodic w.r.t. the fixed-point subalgebra} if it satisfies one of the equivalent properties 
(i)-(vi) listed in Theorem 2.1 in \cite{FM3} (see also \cite{AD}, Definition 3.3). 

In particular, $(\ga,\F)$ is uniquely ergodic w.r.t. the fixed-point subalgebra if and only if the ergodic averages $\frac1{n}\sum_{k=0}^{n-1}\F^k$ converge, point-wise in norm, 
necessarily to a $\F$-invariant conditional expectation onto $\ga^\F$. Obviously, if $\ga^\F=\bc$, then unique ergodicity w.r.t. the fixed-point subalgebra is the same as unique ergodicity.

Notice that, if $\F$ is merely completely positive, the analogous property of unique ergodicity w.r.t. the fixed-point operator system would concern the convergence, point-wise in norm, of the ergodic averages $\frac1{n}\sum_{k=0}^{n-1}\F^k$ to a
unital completely positive projection onto the operator system $\ga^\F$ which is invariant under $\F$.

Without further mention, from now on we specialise the matter to $C^*$-dynamical systems $(\ga,\F)$ such that $\ga$ is a unital $C^*$-algebra, and $\F:\ga\to\ga$ is a unital $*$-homomorphism.

Define
$$
\s^{\mathop{\rm ph}}_{\mathop{\rm pp}}(\F):=\big\{\l\in\bt\mid \l\,\text{is an eigenvalue of}\,\,\F\big\}
$$
the set of the peripheral eigenvalues of $\F$ ({\it i.e. the peripheral pure-point spectrum}), with $\ga_\l$ the relative eigenspaces. 

Obviously, $\idd\in\ga^\F\equiv\ga_1$ because $\F$ preserves the identity. 
Since $\F$ is a $*$-homomorphism,  
$\s^{\mathop{\rm ph}}_{\mathop{\rm pp}}(\F)=\s^{\mathop{\rm ph}}_{\mathop{\rm pp}}(\F)^{-1}$,
and in addition,  
$$
\big\{\l\in\bt\mid\,\,\ga_\l\,\text{contains an invertible operator}\big\}
$$
is a subgroup of $\s^{\mathop{\rm ph}}_{\mathop{\rm pp}}(\F)$. Indeed, $x\in\ga_\l\Rightarrow x^*\in\ga_{\l^{-1}}$, and if $x_1\in\ga_{\l_1}$,  $x_2\in\ga_{\l_2}$ are invertible, then $x_1x_2\in\ga_{\l_1\l_2}$ is an invertible 
eigenvector.

We also consider the peripheral pure-point spectrum
$$
\s^{\mathop{\rm ph}}_{\mathop{\rm pp}}(V_{\f,\F}):=\big\{\l\in\bt\mid \l\,\text{is an eigenvalue of}\,\,V_{\f,\F}\big\}
$$ 
of the isometry $V_{\f,\F}$.

Let $\f\in\cs(\ga)^\F$ and $E^{\f,\F}_\l\in\cb(\ch_\f)$ be the self-adjoint projection onto the eigenspace of $V_{\f,\F}$ corresponding to $\l\in\bt$. Of course, if 
$\l\notin\s^{\mathop{\rm ph}}_{\mathop{\rm pp}}(V_{\f,\F})$ then $E^{\f,\F}_\l=0$, and $E^{\f,\F}_1$ is projecting onto the closed subspace made 
of the vectors invariant under $V_{\f,\F}$. We also recall that 
$$
\dim\big(E^{\f,\F}_1\big)=1 \Longrightarrow \f\in\partial\left(\cs(\ga)^\F\right)\,,
$$
see {\it e.g.} \cite{Sa}, Proposition 3.1.10.

For $\xi\in\ch_\f$ and $n\in\bz$, consider the sequence
$$
\widehat{\m_\xi}(n):=\bigg\{\begin{array}{ll}
                      \langle V^n_{\f,\F}\xi,\xi\rangle& \text{if}\,\, n\geq0\,, \\[1ex]

                \overline{ \langle V^{-n}_{\f,\F}\xi,\xi\rangle} & \text{if}\,\, n<0\,.
                    \end{array}
                    \bigg.
$$
It is well known ({\it cf.} \cite{SN}, Lemme 1) that, for each $\xi\in\ch_\f$, such a sequence $\big\{\widehat{\m_\xi}(n)\big\}_{n\in\bz}$ is the Fourier transform of a positive bounded Radon measure $\m_\xi$ on the unit circle $\bt$. Therefore, $\m_\xi$ is nothing else than the spectral measure of $V_{\f,\F}$ relative to $\xi\in\ch_\f$:
if $\l=e^{-\imath\th}\in\bt$ then $\m_\xi(\{\th\})=\|E^{\f,\F}_\l\xi\|^2$.

One of the key-object of the present analysis is the
{\it full peripheral pure-point spectrum}
$$
\s^{(\mathop{\rm ph, f})}_{\mathop{\rm pp}}(\F):=\bigcup\big\{\s^{\mathop{\rm ph}}_{\mathop{\rm pp}}(V_{\f,\F})\mid \f\in\cs(\ga)^\F\big\}\,,
$$ 
where the suffix "${\rm f}$\," stands for "${\rm full}$\,".

Notice that, if $\F$ is a $*$-automorphism and $\f$ is an invariant state, $\s(\F),\s(V_{\f,\F})\subset\bt$, and therefore we simply write $\s_{\rm pp}(\F)$ and $\s^{({\rm f})}_{\rm pp}(\F)$. 

In addition, if $(\ga,\F)$ is uniquely ergodic with $\f$ the unique invariant state, then 
$\s^{(\mathop{\rm ph, f})}_{\mathop{\rm pp}}(\F)=\s^{\mathop{\rm ph}}_{\mathop{\rm pp}}(V_{\f,\F})$.

\section{$C^*$-dynamical systems, general properties}

For the sake of completeness, we start by reporting some standard results which were proved in \cite{F4}.
\begin{prop}
\label{gnsc}
Let the $C^*$-dynamical system $(\ga,\F,\f)$ be uniquely ergodic. Then $\s^{\mathop{\rm ph}}_{\mathop{\rm pp}}(\F)$ is a subgroup of $\bt$, and the corresponding eigen\-spaces $\ga_\l$, 
$\l\in\s^{\mathop{\rm ph}}_{\mathop{\rm pp}}(\F)$, are generated by a single unitary $u_\l$.
\end{prop}
\noindent\hspace{12.33 cm}$\square$
\vskip.3cm
In general, there is no relation between $\s^{\mathop{\rm ph}}_{\mathop{\rm pp}}(\F)$ and $\s^{(\mathop{\rm ph, f})}_{\mathop{\rm pp}}(\F)$. However, for a uniquely ergodic $C^*$-dynamical system $(\ga,\F,\f)$,
a simple application of the Proposition \ref{gnsc} leads to
$$
\s^{\mathop{\rm ph}}_{\mathop{\rm pp}}(\F)\subset\s^{\mathop{\rm ph}}_{\mathop{\rm pp}}(V_{\f,\F})\equiv\s^{(\mathop{\rm ph, f})}_{\mathop{\rm pp}}(\F)\,.
$$
For some examples of uniquely ergodic $C^*$-dynamical systems, we get
$\s^{\mathop{\rm ph}}_{\mathop{\rm pp}}(\F)\subsetneq\s^{\mathop{\rm ph}}_{\mathop{\rm pp}}(V_{\f,\F})$, see {\it e.g.} \cite{DFGR, F4, R}.

\smallskip

For $\l\in\bt$, suppose that $u\in\ga_\l$ is an isometry. Since $\F$ is a $*$-map, $\l^{-1}\in\s^{(\mathop{\rm ph})}_{\mathop{\rm pp}}(\F)$ with $u^*\in\ga_{\l^{-1}}$ a co-isometry. Therefore non unitary eigenvectors associated to the peripheral spectrum
which are isometries or co-isometries appear in pair: they appear both or do not appear at all. For these cases, we provide a characterisation of whether 
$\l\in\s^{\mathop{\rm ph}}_{\mathop{\rm pp}}(\F)$ is also in 
$\s^{(\mathop{\rm ph, f})}_{\mathop{\rm pp}}(\F)$. 
\begin{prop}
Let $(\ga,\F)$ be a $C^*$-dynamical system, and $\l\in\s^{(\mathop{\rm ph})}_{\mathop{\rm pp}}(\F)$ with $u\in\ga_\l$. 
\begin{itemize}
\item[(i)] If $u$ is an isometry, then $\l\in\s^{(\mathop{\rm ph, f})}_{\mathop{\rm pp}}(\F)$ as well.
\item[(ii)] If $u$ is a co-isometry, then $\l\in\bt\smallsetminus\s^{(\mathop{\rm ph, f})}_{\mathop{\rm pp}}(\F)$ if and only if
for each $\om\in\cs(\ga)^\F$ we have $\xi_\om\perp\pi_\om(u^*u)\ch_\om$.
\end{itemize}
\end{prop}
\begin{proof}
(i) Suppose $u\in\ga_\l$ is an isometry. Then for each invariant state $\om$, $\pi_\om(u)\xi_\om$ is a non null eigenvector of $V_{\om,\F}$ corresponding to the eigenvalue $\l$, and thus $\l\in\s^{(\mathop{\rm ph, f})}_{\mathop{\rm pp}}(\F)$.

(ii) For $\l\in\bt$, suppose that $\l\notin\s^{(\mathop{\rm ph, f})}_{\mathop{\rm pp}}(\F)$ and 
$\om\in\cs(\ga)^\F$. Since $\pi_\om(u)\xi_\om$ is an eigenvector of $V_{\om,\F}$ corresponding to the eigenvalue $\l$, it must be zero. But this simply means that $\xi_\om$ is orthogonal to all vectors corresponding to the initial subspace $\pi_\om(u^*u)\ch_\om$ of $\pi_\om(u)$.

Conversely, suppose $\xi_\om\perp\pi_\om(u^*u)\ch_\om$ for each invariant state $\om$ and $\l\in\s^{(\mathop{\rm ph, f})}_{\mathop{\rm pp}}(\F)$. The latter condition means that there exists an invariant state $\f$ and a norm $1$ eigenvector $\xi\in\ch_\f$ of $V_{\f,\F}$ corresponding to the eigenvalue $\l$. On the other hand, with $\eta:=\pi_\om(u^*)\xi$ we have $\xi:=\pi_\f(u)\eta$. In addition, $\eta$ is invariant for $V_{\f,\F}$, and $\|\eta\|=1$ because $u$ is a co-isometry. Consider the cyclic projection $P\in\pi_\f(\ga)'$ onto the subspace $\pi_\f(\ga)\eta$, together with the vector state $\om:=\langle\pi_\f(\,{\bf\cdot}\,)\eta,\eta\rangle$ on $\ga$ generated by $\eta$. It is invariant because $\eta$ is an invariant vector for $V_{\f,\F}$. Since $P$ commutes also with $V_{\f,\F}$, we recognise that the covariant GNS representation $\big(\ch_\om,\pi_\om, V_{\om,\F},\xi_\om\big)$
coincides, up to unitary equivalence, with $\big(P\ch_\f,P\pi_\f, PV_{\f,\F},\eta\big)$. Therefore, firstly
$$
\xi_\om=\eta=V_{\f,\F}\eta=V_{\f,\F}P\eta=(PV_{\f,\F})\eta=V_{\om,\F}\xi_\om\,.
$$
Secondly, we obtain the contradiction
$$
0\neq\xi=\pi_\f(u)\eta=(P\pi_\f(u))\eta=\pi_\om(u)\xi_\om=0\,,
$$
where the last equality comes from our assumption $\pi_\om(u)\xi_\om=0$.
\end{proof}

The key-point of our analysis is the following simple generalisation of Lemma 2.1 in \cite{R} and Lemma 1 in \cite{F4}, of which we report the details of the proof for the convenience of the reader.
\begin{lem}
\label{mle}
Consider the $C^*$-dynamical system $(\ga,\F)$, and a sequence of states 
$\{\om_n\}_{n\in\bn}\subset\cs(\ga)$. Then for each $a\in\ga$ and $\l=e^{-\imath\th}$, there exists $\om\in\cs(\ga)^\F$ such that
$$
\m_{\pi_\om(a)\xi_\om}(\{\th\})^{1/2}\geq\limsup_n\frac1{n}\bigg|\sum_{k=0}^{n-1}\om_n\big(\F^k(a)\big)\l^{-k}\bigg|\,.
$$
\end{lem}
\begin{proof}
With $\l=e^{-\imath\th}$ and $\bt=[0,2\pi)$, consider the $C^*$-tensor product $C(\bt)\otimes\ga\equiv C(\bt;\ga)$ together with the $*$-homomorphism $\widetilde{\F}: C(\bt;\ga)\to C(\bt;\ga)$ given by
$$
\widetilde{\F}(f)(s):=\F\big(f(s+\th)\big)\,. \quad f\in C(\bt;\ga)\,.
$$

For $\{\om_n\}_{n\in\bn}\subset\cs(\ga)$, let $\{\widetilde{\om}_n\}_{n\in\bn}\subset\cs(C(\bt;\ga))$ be the sequence of states given by
\begin{align*}
\widetilde{\om}_n(f):=&\bigg(\frac1{n}\sum_{k=0}^{n-1}(\d_{0}\otimes \om_n)\circ\widetilde{\F}^k\bigg)(f)\\
=&\bigg(\frac1{n}\sum_{k=0}^{n-1}\d_{k\th}\otimes\big(\om_n\circ\F^k\big)\bigg)(f)\\
=&\frac1{n}\sum_{k=0}^{n-1}\om_n\big(\F^k(f(k\th))\big)\,.
\end{align*}
Notice that for the function $g(s):=ae^{\imath s}\in C(\bt;\ga)$,
$$
\widetilde{\om}_n(g)=\frac1{n}\sum_{k=0}^{n-1}\om_n\big(\F^k(a)\big)\l^{-k}\,.
$$

Let $\{n_j\}_{j\in\bn}\subset\bn$ be a subsequence such that
$$
\limsup_n\frac1{n}\bigg|\sum_{k=0}^{n-1}\om_n\big(\F^k(a)\big)\l^{-k}\bigg|
=\lim_j\frac1{n_j}\bigg|\sum_{k=0}^{n_j-1}\om_{n_j}\big(\F^k(a)\big)\l^{-k}\bigg|\,,
$$
and consider any $*$-weak limit point $\widetilde{\om}$ of the sequence $\{\widetilde{\om}_{n_j}\}_{j\in\bn}$ which exists by the Banach-Alaoglu Theorem ({\it e.g.} \cite{RS}, Theorem IV.21). 
By passing to a subsequence if necessary,
we get
$$
\big|\widetilde{\om}(f)\big|=\lim_j\frac1{n_j}\bigg|\sum_{k=0}^{n_j-1}\om_{n_j}\big(\F^k(a)\big)\l^{-k}\bigg|\,.
$$

Let $\om\in\cs(\ga)$ be the marginal of $\widetilde{\om}$ defined on constant functions $f_a(s):=a$ by 
$$
\om(a):=\widetilde{\om}(f_a)\,,\quad a\in\ga\,.
$$
By construction, $\widetilde{\om}$ is invariant under $\widetilde{\F}$. Therefore, $\om$ is invariant under $\F$ as well.

Let $\big(\ch_{\widetilde{\om}},V_{\widetilde{\om},\widetilde{\F}},\pi_{\widetilde{\om}},\xi_{\widetilde{\om}}\big)$ be the covariant GNS representation associated to $\widetilde{\om}$.
By computing as in Lemma 2.1 of \cite{R}, we then conclude for the spectral measures associated to $V_{\widetilde{\om},\widetilde{\F}}$ and $V_{\om,\F}$,
$$
\m_{\pi_{\widetilde{\om}}(f)\xi_{\widetilde{\om}}}(\{0\})=\m_{\pi_\om(a)\xi_\om}(\{\th\})\,.
$$
Therefore, with $P_{\rm const}\in\cb(\ch_{\widetilde{\om}})$ the orthogonal projections onto the one dimensional subspace $\bc\xi_{\widetilde\om}$,
\begin{align*}
\m_{\pi_\om(a)\xi_\om}(\{\th\})^{1/2}=&\m_{\pi_{\widetilde{\om}}(f)\xi_{\widetilde{\om}}}(\{0\})^{1/2}=\big\|E^{\widetilde{\om},\widetilde{\F}}_1\pi_{\widetilde{\om}}(f)\xi_{\widetilde{\om}}\big\|
\geq\|P_{\rm const}\big(\pi_{\widetilde{\om}}(f)\xi_{\widetilde{\om}}\big)\|\\
=&\big|\widetilde{\om}(f)\big|
=\limsup_n\frac1{n}\bigg|\sum_{k=0}^{n-1}\om_n\big(\F^k(a)\big)\l^{-k}\bigg|\,.
\end{align*}
\end{proof}
The main result of the present section is the following
\begin{thm}
\label{muewrt}
Let $(\ga,\F)$ be a $C^*$-dynamical system. Fix 
$\l\in\bt\smallsetminus\s^{(\mathop{\rm ph, f})}_{\mathop{\rm pp}}(\F)$. Then for each $a\in\ga$,
$$
\lim_{n\to+\infty}\frac1{n}\sum_{k=0}^{n-1}\l^{-k}\F^k(a)=0\,,
$$
in the norm topology of $\ga$.
\end{thm}
\begin{proof}
Let $\l\in\bt\smallsetminus\s^{(\mathop{\rm ph, f})}_{\mathop{\rm pp}}(\F)$, and suppose that there exists $a\in\ga$ such that $\frac1{n}\sum_{k=0}^{n-1}\F^k(a)\l^{-k}\nrightarrow0$ in the norm topology. Then it would exist a sequence of states  
$\{\om_n\}_{n\in\bn}\subset\cs(\ga)$ such that $\limsup_n\frac1{n}\left|\sum_{k=0}^{n-1}\om_n\big(\F^k(a)\big)\l^{-k}\right|>0$. By Lemma \ref{mle}, for $\l=e^{-\imath\th}$ we would find an invariant state $\om$ such that
$$
\m_{\pi_\om(a)\xi_\om}(\{\th\})^{1/2}\geq\limsup_n\frac1{n}\bigg|\sum_{k=0}^{n-1}\om_n\big(\F^k(a)\big)\l^{-k}\bigg|>0\,,
$$
which would contradict $\l\notin\s^{(\mathop{\rm ph, f})}_{\mathop{\rm pp}}(\F)$.
\end{proof}

\section{uniquely ergodic $C^*$ dynamical systems with respect to the fixed-point subalgebra}

In the present section, we study the convergence of Cesaro averages $M_{a,\l}$, $\l\in\bt$, in \eqref{malam} for $C^*$-dynamical systems $(\ga,\F)$, made of a unital $C^*$-algebra and a unital $*$-homomorphism, enjoying the property of unique ergodicity w.r.t. the fixed-point subalgebra. 
  
The convergence of such averages $M_{a,\l}$ is in general not granted, even for  $C^*$ dynamical systems enjoying such a strong ergodicity property when, for the fixed-point $*$-subalgebra, $\ga_1\supsetneq\bc\idd$. However, we can provide some useful criteria which assure the convergence, and exhibit examples for which such results apply.
\begin{thm}
\label{isoco}
Let $(\ga.\F)$ be a uniquely ergodic $C^*$-dynamical system w.r.t. the fixed-point subalgebra. For $\l\in\s^{\mathop{\rm ph}}_{\mathop{\rm pp}}(\F)$, suppose that 
$u\in\ga_\l$ is an isometry or a co-ismetry. Then
\begin{itemize}
\item[(i)] $\ga\ni x\mapsto E_\l(x):=E_1(xu^*)u\,\,\in\ga_\l$ (isometry case)
\item[(ii)] $\ga\ni x\mapsto E_\l(x):=uE_1(u^*x)\,\,\in\ga_\l$ (co-isometry-case)
\end{itemize}
uniquely define a norm-one projection $E_\l:\ga\rightarrow\ga_\l$ which is independent on the choice of $u$ among the isometries and co-isometries of $\ga_\l$.

In addition, 
$$
\lim_{n\to+\infty}\bigg(\frac1{n}\sum_{k=0}^{n-1}\l^{-k}\F^k(x)\bigg)=E_\l(x)\,,
$$
in the norm topology.
\end{thm}
\begin{proof}
Suppose that $u\in\ga_\l$ is an isometry. We get
\begin{align*}
E_1(xu^*)u=&\lim_n\bigg(\frac1{n}\sum_{k=0}^{n-1}\F^k(xu^*)\bigg)u
=\lim_n\bigg(\frac1{n}\sum_{k=0}^{n-1}\F^k(x)\F^k(u^*)\bigg)u\\
=&\lim_n\bigg(\frac1{n}\sum_{k=0}^{n-1}\l^{-k}\F^k(x)u^*\bigg)u
=\lim_n\bigg(\frac1{n}\sum_{k=0}^{n-1}\l^{-k}\F^k(x)\bigg)u^*u\\
=&\lim_n\bigg(\frac1{n}\sum_{k=0}^{n-1}\l^{-k}\F^k(x)\bigg)\,.
\end{align*}

Since the r.h.s. does not depend on the isometry $u\in\ga_\l$, the l.h.s. gives rise to a linear contraction independent on the choice of the isometry $u$. The same holds true when $u$ is a co-isometry with 
$E_\l=uE_1(u^*\,{\bf\cdot}\,)$.

Now we show that $E_\l$, which is defined provided $\ga_\l$ contains either an isometry or a co-isometry as we have just shown, is a projection onto $\ga_\l$, and thus $\|E_\l\|=1$. Indeed, we firstly suppose $x\in\ga_\l$, then
\begin{align*}
E_\l(x)=\lim_n\bigg(\frac1{n}\sum_{k=0}^{n-1}\l^{-k}\F^k(x)\bigg)=\lim_n\bigg(\frac1{n}\sum_{k=0}^{n-1}\l^{-k}\l^{k}\bigg)x=x\,.
\end{align*}
Secondly, for $x\in\ga$, put $y:=E_\l(x)\in\ga_\l$. By the last calculation, 
$$
E_\l(E_\l(x))=E_\l(y)=y=E_\l(x)\,,
$$
and the proof is complete.
\end{proof}
Now we list some immediate consequences of the previous theorem.
\begin{prop}
\label{muewrt0}
Let $\l\in\s^{\mathop{\rm ph}}_{\mathop{\rm pp}}(\F)$ and $u\in\ga_\l$ be an isometry or a co-isometry. Then we get
\begin{itemize}
\item[(i)] isometry case: $u^*E_1(uu^*)u=\idd$, and $E_1(u^*)u=0$ if $\l\neq1$;
\item[(ii)] co-isometry case: $uE_1(u^*u)u^*=\idd$, and $uE_1(u^*)=0$ if $\l\neq1$.
\end{itemize}
\end{prop}
\begin{proof}
If $u\in\ga_\l$ is an isometry, we get 
$u=E_\l(u)=E_1(uu^*)u$, and thus $\idd=u^*E_1(uu^*)u$ after multiplying both members by $u^*$ from the left. 

Concerning the second assertion,
$$
0=\lim_n\bigg(\frac1{n}\sum_{k=0}^{n-1}\l^{-k}\bigg)\idd=\lim_n\bigg(\frac1{n}\sum_{k=0}^{n-1}\l^{-k}\F(\idd)^k\bigg)=E_1(u^*)u\,.
$$

The case of a co-isometry $u\in\ga_\l$ follows analogously.
\end{proof}
\begin{rem}
If $(\ga,\F,\f)$ is uniquely ergodic and $\l\in\s^{\mathop{\rm ph}}_{\mathop{\rm pp}}(\F)$ with $u\in\ga_\l$, then:
\begin{itemize}
\item[(i)] if $\l\neq1$ then $\f(u)=0$,
\item[(ii)] $\f(xu^*)u=\f(u^*x)u$, $x\in\ga$.
\end{itemize}
\end{rem}
\noindent\hspace{12.33 cm}$\square$
\begin{prop}
Let $(\ga,\F)$ be a $C^*$-dynamical system, uniquely ergodic w.r.t. the fixed-point subalgebra. Suppose that $\l\in\s^{\mathop{\rm ph}}_{\mathop{\rm pp}}(\F)$, where
$u\in\ga_\l$ is an isometry or a co-isometry. Then $\l\in\s^{(\mathop{\rm ph, f})}_{\mathop{\rm pp}}(\F)$ as well.
\end{prop}
\begin{proof}
The case of the isometry is trivial because, for each $\f\in\cs(\ga)^\F$, $\l$ is an eigenvalue of $V_{\f,\F}$ corresponding to the eigenvector $\pi_\f(u)\xi_\f$ and
$$
\|\pi_\f(u)\xi_\f\|^2=\langle\pi_\f(u)\xi_\f,\pi_\f(u)\xi_\f\rangle=\langle\pi_\f(u^*u)\xi_\f,\xi_\f\rangle=\|\xi_\f\|^2=1\,.
$$
Therefore, $\l\in\s^{(\mathop{\rm ph, f})}_{\mathop{\rm pp}}(\F)$.

Suppose now that $u\in\ga_\l$ is a co-isometry and $\l\notin\s^{(\mathop{\rm ph, f})}_{\mathop{\rm pp}}(\F)$. Then, by Theorem \ref{muewrt}, we obtain the contradiction
$$
0=\lim_n\bigg(\frac1{n}\sum_{k=0}^{n-1}\l^{-k}\F^k(u)\bigg)u^*=\lim_n\bigg(\frac1{n}\sum_{k=0}^{n-1}\l^{-k}\l^k\bigg)uu^*=uu^*=\idd\,.
$$
\end{proof}
We end by noticing that all above results, that is for example the existence of the projection $E_\l$ whenever $\l\in\s^{\mathop{\rm ph}}_{\mathop{\rm pp}}(\F)$, can be extended to the case when $\ga_\l$ contains an invertible operator $a$. Namely, for $a\in\ga$ invertible we can check $\F(a^{-1})=\F(a)^{-1}$ because $\F$ is multiplicative and 
identity-preserving. Hence,  
$a\in\ga_\l$ invertible implies $0\neq a^{-1}\in\ga_{\l^{-1}}$. Therefore, we easily get
$$
aE_1(a^{-1}x)=\lim_{n\to+\infty}\bigg(\frac1{n}\sum_{k=0}^{n-1}\l^{-k}\F^k(x)\bigg)=E_1(xa^{-1})a\,.
$$
The same conclusion can be easily obtained if, for $a_j,b_j\in\ga_\l$, $\idd=\sum_{j=1}^m a_j^*b_j$:
$$
\lim_{n\to+\infty}\bigg(\frac1{n}\sum_{k=0}^{n-1}\l^{-k}\F^k(x)\bigg)=\sum_{j=1}^mE_1(xa_j^*)b_j\,,
$$
or 
$$
\lim_{n\to+\infty}\bigg(\frac1{n}\sum_{k=0}^{n-1}\l^{-k}\F^k(x)\bigg)=\sum_{j=1}^ma_jE_1(b_j^*x)
$$
whenever $\idd=\sum_{j=1}^m a_jb^*_j$. 

The case involving infinite sums might provide the result as well, after solving some technical problems. We hope to return somewhere else on the question relative to the existence, under more general conditions, of the projections $E_\l$.

\section{some examples}

We describe some manageable examples, perhaps of certain interest for the applications to quantum probability, to which the previous results can be applied.
The reader is referred to \cite{CrF2, CrF, F28} for further details. Some other examples arising from the noncommutative 2-torus ({\it e.g.} \cite{DFGR}) will be presented elsewhere.

\subsection{The monotone case}
We consider the $C^*$-dynamical system $(\gpm,s)$ where $\gpm$ is the concrete $C^*$-algebra generated by the identity $I=\idd_\gpm$ and the monotone creators $\{m^\dagger_n\mid n\in\bz\}$ acting on the monotone Fock space $\G_{\rm mon}(\ell^2(\bz))$ on $\ell^2(\bz)$. It has the structure $\gpm=\gpa+\bc I$, where $\gpa$ is the non unital $C^*$-algebra generated by the monotone creators. Therefore, $I\notin\gpa$ and thus the state at infinity $\om_\infty$ is meaningful.
The one-step shift $s$ is defined on generators as $s(m^\dagger_j)=m^\dagger_{j+1}$, $j\in\bz$.

The main properties of $(\gpm,s)$ are summarised as follows:
\begin{itemize}
\item[{\bf-}] for the fixed-point $*$-subalgebra, $\gpm^s=\bc I$,
\item[{\bf-}]  the set of all invariant states 
$$
\cs(\gpm)^s=\big\{(1-t)\om_\gpo+t\om_\infty\mid t\in[0,1]\big\}
$$ 
is the convex combination of the vacuum state $\om_\gpo$ and the state at infinity $\om_\infty$.
\end{itemize}
Therefore, $(\gpm,s)$ cannot be uniquely ergodic w.r.t. the fixed-point subalgebra. Indeed, it can be viewed by direct inspection because
$$
\frac1{n}\sum_{k=0}^{n-1}s^k(m_lm^\dagger_l)=\frac1{n}\sum_{k=0}^{n-1}m_{l+k}m^\dagger_{l+k}\downarrow P_{e_\gpo}\,,
$$
the self-adjoint projection onto the subspace generated by the vacuum vector $e_\gpo$. Such a convergence in the strong operator topology, cannot be in norm.

Obviously, $1$ is contained in both spectra $\s_{\mathop{\rm pp}}(s)$, $\s^{(\rm f)}_{\mathop{\rm pp}}(s)$, and in addition $\s_{\mathop{\rm pp}}(s)=1$. We now show that 
$\s^{(\rm f)}_{\mathop{\rm pp}}(s)=1$, a fact which directly follows by applying the previous Theorem \ref{muewrt}.
\begin{prop}
Let $\l\in\bt\smallsetminus\{1\}$. Then for each $x\in\gpm$,
$$
\lim_{n\to+\infty}\frac1{n}\sum_{k=0}^{n-1}\l^{-k}s^k(x)=0\,,
$$
in the norm topology.
\end{prop}
\begin{proof}
We start by recalling that $\cs(\gpm)^s$ is the convex combination of the vacuum state and the state at infinity.
For the latter, its GNS covariant representation 
$\big(\ch_{\om_\infty},\pi_{\om_\infty}, V_{\om_\infty,s},\xi_{\om_\infty}\big)$ is nothing but the trivial one $(\bc,\pi,U_1,1)$ on $\bc$, where $\pi(a+bI)=b$ and $U_1$ is the unitary given by the multiplication by 1. Consequently, $\s_{\mathop{\rm pp}}\big(V_{\om_\infty,s}\big)=\{1\}$.

Let $u$ be the unitary implementing the shift on the one-particle subspace $\ell^2(\bz)$.
Define, $V=\G_{\rm mon}(u)$ as $\G_{\rm mon}(u)e_\gpo:=e_\gpo$,
$$
\G_{\rm mon}(u)e_{j_1}\otimes\cdots\otimes e_{j_n}=ue_{j_1}\otimes\cdots\otimes ue_{j_n}=e_{j_1+1}\otimes\cdots\otimes e_{j_n+1}\,.
$$
It is a well defined (because it transforms the increasing sequence $(j_1,\cdots,j_n)$ to the increasing one $(j_1+1,\cdots,j_n+1)$) unitary operator acting on the monotone Fock space implementing the shift automorphism $s$ on $\gpm$. 

Concerning the GNS representation $\big(\ch_{\om_\gpo},\pi_{\om_\gpo}, V_{\om_\gpo,s},\xi_{\om_\gpo}\big)$ of the vacuum state, we get $\ch_{\om_\gpo}=\G_{\rm mon}\big(\ell^2(\bz)\big)$, the monotone Fock space, $\pi_{\om_\gpo}=\id_\gpm$, the identical representation of $\gpm$ on the monotone Fock space, $V_{\om_\gpo,s}=\G_{\rm mon}(u)$, the monotone second quantisation of $u$, and finally $\xi_{\om_\gpo}=e_\gpo$, the vacuum vector. Since $\s_{\mathop{\rm pp}}(u)=\{1\}$, we argue that $\s_{\mathop{\rm pp}}\big(V_{\om_\gpo,s}\big)=\{1\}$ as well.

For $t\in(0,1)$, let $\f_t:=(1-t)\om_\gpo+t\om_\infty$. Its GNS covariant representation $\big(\ch_{\f_t},\pi_{\f_t}, V_{\f_t,s},\xi_{\f_t}\big)$ is easily obtained by $\ch_{\f_t}=\ch_{\om_\gpo}\oplus\ch_{\om_\infty}$,
$\pi_{\f_t}=\pi_{\om_\gpo}\oplus\pi_{\om_\infty}$, $V_{\f_t,s}=V_{\om_\gpo,s}\oplus V_{\om_\infty,s}$, $\xi_{\f_t}=\sqrt{1-t}\xi_{\om_\gpo}\oplus\sqrt{t}\xi_{\om_\infty}$. Therefore, 
$\s_{\mathop{\rm pp}}\big(V_{\f_t,s}\big)=\{1\}$, and consequently, $\s^{(\rm f)}_{\mathop{\rm pp}}(s)=\{1\}$.

The proof now follows from Theorem \ref{muewrt}.
\end{proof}
We can get the above result also by a direct computation.  Namely, by using \cite{CrF2}, Theorem 3.4, we can reduce the matter when $x\not=\a I$ is one of the words generating $\gpa\subsetneq\gpa+\bc I=\gpm$. For all words $a\in\gpa$ in normal order, by reasoning as in \cite{CrF}, Proposition 4.2, we conclude that $\frac1{n}\sum_{k=0}^{n-1}\l^{-k}s^k(a)\rightarrow0$, uniformly for each $\l\in\bt$. It remains the case when $a=m_lm^\dagger_l$, $l\in\bz$. 
For such cases and for $\l\in\bt\smallsetminus\{1\}$, we have to compute $\big\langle\frac1{n}\sum_{k=0}^{n-1}\l^{-k}s^k(a)\xi,\eta\big\rangle$ for unit vectors $\xi,\eta\in\G_{\rm mon}\big(\ell^2(\bz)\big)$. 

After some straightforward computations, we get
$$
\bigg|\bigg\langle\frac1{n}\sum_{k=0}^{n-1}\l^{-k}s^k(a)\xi,\eta\bigg\rangle\bigg|\leq\frac4{n|\l-1|}\rightarrow0\,,
$$
uniformly for $\xi,\eta$ in the unit ball of $\G_{\rm mon}\big(\ell^2(\bz)\big)$.

\subsection{The boolean case}
We consider the $C^*$-dynamical system $(\gpb,s)$, where $\gpb$ is the concrete $C^*$-algebra generated by the identity and the boolean creators $\{b^\dagger_n\mid n\in\bz\}$ acting on the boolean Fock space 
$\G_{\rm boole}(\ell^2(\bz))$ on $\ell^2(\bz)$, and (with an abuse of notation) $s$ is the one-step shift acting on generators as $s(b^\dagger_j)=b^\dagger_{j+1}$, $j\in\bz$.

In \cite{CrF}, it was shown that $\gpb$ is nothing but the $C^*$-algebra $\ck\big(\ell^2(\{\gpo\}\sqcup\bz\big)+\bc I$ generated by all compact operators acting on 
$\G_{\rm boole}(\ell^2(\bz))=\ell^2\big(\{\gpo\}\sqcup\bz\big)$ and the identity 
$I:=\idd_{\ell^2(\{\gpo\}\sqcup\bz)}$. The shift is therefore generated by the adjoint action $\ad_V$, with $V$ defined on the canonical basis $\{e_\gpo\}\sqcup\{e_j\mid j\in\bz\}$ of $\ell^2\big(\{\gpo\}\sqcup\bz\big)$
by
$$
Ve_\gpo=e_\gpo\,,\quad Ve_j=e_{j+1}\,,\,\,j\in\bz\,.
$$
By following \cite{CrF}, Section 7, we have:
\begin{itemize}
\item[{\bf-}] for the fixed-point $*$-subalgebra, $\gpb_1\equiv\gpb^s=\bc P_{e_\gpo}\bigoplus\bc P_{e_\gpo}^\perp$;
\item[{\bf-}]  the set of all invariant states 
$$
\cs(\gpb)^s=\big\{(1-t)\om_\gpo+t\om_\infty\mid t\in[0,1]\big\}
$$ 
is the convex combination of the vacuum state $\om_\gpo$ and the state at infinity $\om_\infty$;
\item[{\bf-}] with $a\in\ck\big(\ell^2(\{\gpo\}\sqcup\bz\big)$, 
$$
\gpb\ni A+bI\mapsto \ce_1(a+bI):=\big(\langle Ae_\gpo,e_\gpo\rangle+b\big)P_{e_\gpo}+bP_{e_\gpo}^\perp\in\gpb_1
$$
is a conditional expectation, invariant under the shift $s$;
\item[{\bf-}] the $C^*$-dynamical system $(\gpb,s)$ is uniquely mixing (hence uniquely ergodic, w.r.t. the fixed-point subalgebra, {\it cf.} \cite{F1}) w.r.t. the conditional expectation 
$\ce_1$.
\end{itemize}
Notice that the set $\cs(\gpb)^s$ of the boolean invariant states has the same structure as that $\cs(\gpm)^s$ of the monotone invariant ones. Furthermore,
$\s_{\mathop{\rm pp}}(u)=\{1\}=\s^{({\rm f})}_{\mathop{\rm pp}}(u)$ 
as for the monotone case. Differently to $(\gpm,s)$, the $C^*$-dynamical system $(\gpb,s)$ is uniquely ergodic w.r.t. the fixed-point subalgebra. Therefore, for the convergence of ergodic averages we have for 
$x\in\gpb$ and $\l\in\bt$,
$$
\lim_{n\to+\infty}\frac1{n}\sum_{k=0}^{n-1}\l^{-k}s^k(x)=\bigg\{\begin{array}{ll}
                      \ce_1(x)& \text{if}\,\, \l=1\,, \\[1ex]

                0 & \text{if}\,\, \l\neq1\,.
                    \end{array}
                    \bigg.
$$ 
In order to provide an example for which the involved spectra are non trivial, we consider the tensor product construction of the previous boolean $C^*$-dynamical system with the irrational rotations on the unit circle.

For the irrational number $\th\in(0,1)$, consider  the rotation $R_\th$ on $\bt$ of the angle $2\pi\th$: $R_\th(z):=e^{2\pi\imath\th}z$. 
Let $(\ga,\a)$ be the tensor product $C^*$-dynamical system, where $\ga=C(\bt)\otimes\gpb=C\big(\bt; \gpb\big)$,
$$
\a(f)(z):=s\big(f(e^{2\pi\imath\th}z)\big)\,,\quad z\in\bt\,,\,\,f\in C\big(\bt; \gpb\big)\,,
$$
Finally, define
$$
E_1(f):=\bigg(\!\int\!\otimes\,\ce_1\!\bigg)(f)=\oint\ce_1\big(f(z)\big)\frac{\di z}{2\pi\imath z}\,,\quad f\in C\big(\bt; \gpb\big)\,.
$$
Notice that, with $1\in C(\bt)$ the constant function identically equal to 1, $E_1$ is projecting onto the fixed-point $*$-subalgebra $\ga_1=\bc1\otimes\gpb_1\sim\gpb_1$.
\begin{prop}
The $C^*$-dynamical system $(\ga,\a)$ is uniquely ergodic w.r.t. the fixed-point subalgebra with expectation $E_1$. 

In addition, 
$$
\s_{\mathop{\rm pp}}(\a)=\big\{e^{2\pi\imath l\th}\mid l\in\bz\big\}=\s^{(\rm f)}_{\mathop{\rm pp}}(\a)\,,
$$
where, for $\l_l=e^{2\pi\imath l\th}\in\s_{\mathop{\rm pp}}(\a)$,
$\ga_{\l_l}=u_l\ga_1=\ga_1u_l$,
with $u_l(z)=z^l\otimes I\in\ga_{\l_l}$ unitary.

Finally, for $f\in\ga$ and $\l\in\bt$,
$$
\lim_{n\to+\infty}\frac1{n}\sum_{k=1}^{n-1}\l^{-k}\a^k(f)=\bigg\{\begin{array}{ll}
                     \left(\oint\ce_1\big(f(z)\big)\frac{\di z}{2\pi\imath z^{l+1}}\right)u_l& \text{if}\,\, \l=\l_l\,, \\[1ex]

                0 & \text{if}\,\, \l\neq\l_l\,.
                    \end{array}
                    \bigg.
$$
\end{prop}
\begin{proof}
By a standard approximation argument, we can reduce the matter to a finite linear combination of generators of $\ga$ of the form $x=f\otimes a$, where $f\in C(\bt)$ and $a\in\cb\big(\ell^2\big(\{\gpo\}\sqcup\bz\big)\big)$ is a rank-one operator of the form $\langle\,{\bf\cdot}\,,e_i\rangle e_j$, $i,j\in\{\gpo\}\cup\bz\}$, or $a=I$. In the latter case, we simply get 
\begin{align*}
\lim_{n\to+\infty}&\bigg(\frac1{n}\sum_{k=0}^{n-1}\a^k(x)\bigg)=\bigg[\lim_n\bigg(\frac1{n}\sum_{k=0}^{n-1}f\circ R^k_\th\bigg)\bigg]I\\
=&\bigg(\oint f(z)\frac{\di z}{2\pi\imath z}\bigg)I=E_1(x)\,,
\end{align*}
point-wise in norm, because of the unique ergodicity of the irrational rotations on the unit circle. The same happens if $a=P_\gpo$, because it is also invariant under the shift:
$$
\lim_{n\to+\infty}\bigg(\frac1{n}\sum_{k=0}^{n-1}\a^k(x)\bigg)=\bigg(\oint f(z)\frac{\di z}{2\pi\imath z}\bigg)P_\gpo=E_1(x)\,.
$$

Suppose now that $a=\langle\,{\bf\cdot}\,,e_\gpo\rangle e_j$, with $j\in\bz$. By reasoning as in the proof of Proposition 7.2 of \cite{CrF}, for a unit vector $\xi\in\ell^2(\{\gpo\}\cup\bz)$
we have 
$$
\bigg\|\sum_{k=0}^{n-1}f\big(R^k_\th z\big)s^k(a)\xi\bigg\|_{\ell^2(\{\gpo\}\cup\bz)}
=\sqrt{\sum_{k=1}^n\big|f\big(R^k_\th z\big)\langle\xi,e_\gpo\rangle\big|^2}\leq\sqrt{n}\|f\|_\infty\,.
$$
By taking the adjoint in the above estimate, the same holds true for $a=\langle\,{\bf\cdot}\,,e_j\rangle e_\gpo$. Finally, the above estimate also holds true for $a=\langle\,{\bf\cdot}\,,e_i\rangle e_j$, $i,j=\bz$. We therefore conclude for each rank-one operator $a$ as before,
$$
\bigg\|\frac1{n}\sum_{k=0}^{n-1}f\big(R^k_\th z\big)s^k(a)\bigg\|_{\cb(\ell^2(\{\gpo\}\cup\bz))}\leq\frac{\|f\|_\infty}{\sqrt{n}}\,,\quad z\in\bt\,.
$$

Collecting together, we get
$$
\bigg\|\frac1{n}\sum_{k=0}^{n-1}\a^k(x)\bigg\|_\ga
=\max_{z\in\bt}\bigg\|\frac1{n}\sum_{k=0}^{n-1}f\big(R^k_\th z\big)s^k(a)\bigg\|_{\cb(\ell^2(\{\gpo\}\cup\bz))}
\leq\frac{\|f\|_\infty}{\sqrt{n}}\rightarrow0
$$
when $n\to+\infty$, that is 
$$
\lim_{n\to+\infty}\bigg(\frac1{n}\sum_{k=0}^{n-1}\a^k(x)\bigg)=0=E_1(x)\,.
$$

Concerning the spectra, we have for the operator $V$ implementing the shift on $\G_{\rm boole}$, $\s_{\mathop{\rm pp}}(V)=\{1\}$, see {\it e.g.} \cite{F3}, Section 6. So
$$
\s_{\mathop{\rm pp}}(s)=\s_{\mathop{\rm pp}}\big(\!\ad{}\!_V\!\big)=\{1\}\,,
$$
and therefore
$$
\s_{\mathop{\rm pp}}(\a)=\s_{\mathop{\rm pp}}(R_\th)\,.
$$

As in the monotone case, for the general invariant state $\psi_t=\int\otimes\,\f_t$, we have
$$
\s_{\mathop{\rm pp}}(V_{\psi_t,\a})=\big\{e^{2\pi\imath k\th}\mid k\in\bz\big\}=\s^{(\rm f)}_{\mathop{\rm pp}}(\a)\,,\quad t\in[0,1]\,,
$$
and thus 
$$
\s_{\mathop{\rm pp}}(\a)=\s^{({\rm f})}_{\mathop{\rm pp}}(\a)\,.
$$

The last assertion now follows from Theorems \ref{isoco} and \ref{muewrt}.
\end{proof}

\section*{Acknowledgements}

The author is grateful to R. Floricel for the kind invitation to the University of Regina, where this work was completed. He also acknowledges the support of Italian INDAM-GNAMPA.

The present project is part of:
\begin{itemize}
\item[-] OAAMP - Algebre di operatori e applicazioni a
strutture non commutative in matematica e fisica, CUP E81I18000070005;
\item[-] MIUR Excellence Department Project awarded to the Department of Mathematics, University of Rome Tor Vergata, CUP E83C18000100006.
\end{itemize}


\begin{thebibliography}{9999}
    
\bibitem{AD} Abadie B., Dykema K.
{\it Unique ergodicity of free shifts and some other automorphisms of $C^*$-algebras}, J. Operator Theory {\bf 61} (2009), 279-294.

\bibitem{BF2} Barreto S. D., Fidaleo F. 
{\it Disordered Fermions on lattices and their spectral properties}, 
J. Stat. Phys. {\bf 143} (2011), 657-684.

\bibitem{CrF2} Crismale V., Fidaleo F., Griseta M. E. {\it Wick order, spreadability and exchangeability for monotone commutation relations}, Ann. Henri Poincare, 
{\bf 19} (2018), 3179-3196

\bibitem{CrF} Crismale V., Fidaleo F., Lu Y. G. {\it Ergodic theorems in quantum probability: an application to the monotone stochastic processes},  Ann.  Sc.  Norm.  Sup.
Pisa Cl.  Sci. {\bf 17} (2017), 113-141 

\bibitem{DFGR} Del Vecchio S., Fidaleo F., Giorgetti L., Rossi S. {\it Ergodic properties of the Anzai skew-product for the noncommutative torus}, Ergod. Th. Dyn. Syst., published online  (doi:10.1017/etds.2019.116).

\bibitem{F} Fidaleo F. 
{\it On the entangled ergodic theorem},
Infin. Dimens. Anal. Quantum Probab. 
Relat. Top. {\bf 10} (2007), 67-77.

\bibitem{F1} Fidaleo F. {\it On strong ergodic properties of quantum dynamical systems},
Infin. Dimens. Anal. Quantum Probab. 
Relat. Top. {\bf 12} (2009), 551-564.

\bibitem{F2} Fidaleo F. 
{\it The entangled ergodic theorem in the almost periodic case}, 
Linear Algebra Appl., {\bf 432} (2010), 526-535.

\bibitem{F3} Fidaleo F. {\it Nonconventional ergodic theorems for quantum dynamical systems}, Infin. Dimens. Anal. Quantum Probab. Relat. Top.,  
{\bf 17} (2014), 1450009 (21 pages).

\bibitem{F28} Fidaleo F. {\it A note on Boolean stochastic processes}, Open Sys. Inform. Dyn., {\bf22} (2015), 1550004 (10 pages).

\bibitem{F4} Fidaleo F. {\it Uniform convergence of Cesaro averages for uniquely ergodic $C^*$-dynamical systems}, Entropy,  
{\bf 20} (2018), 987 (9 pages).

\bibitem{FM3} Fidaleo F., Mukhamedov F. {\it Strict weak mixing of
some $C^*$-dynamical systems based on free shifts},
J. Math. Anal. Appl. {\bf 336} (2007), 180-187.

\bibitem{NSZ} Niculescu C. P., Str\"oh A., Zsid\'o L.
{\it Noncommutative estension of classical and multiple recurrence 
theorems}, J. Operator Theory {\bf 50} (2003), 3--52.

\bibitem{P} Paulsen V.
{\it Completely bounded maps and operator algebras}, Cambridge University Press, Cambridge-New York-Melbourne, 2002.

\bibitem{RS} Reed M., Simon B.
{\it Functional analysis}, Academic Press,
New York--London 1980.

\bibitem{R} Robinson, E. A. Jr. {\it On uniform convergence in the Wiener-Wintner theorem}, J. London Math. Soc. (2) {\bf 49} (1994), 493-501. 

\bibitem{Sa} Sakai S. \emph{$C^*$-Algebras and $W^*$-Algebras}, Springer-Verlag, Berlin, 1971.

\bibitem{SN} Sz.-Nagy B. {\it Transformations de l’espace de Hilbert, fonctions de type positif sur un groupe}, 
Acta Sci. Szeged {\bf 15} (1954), 104-114.


\end{thebibliography}
\end{document}